\documentclass[12pt]{amsart}
\usepackage{amstext,amsfonts,amsmath,amssymb,amsthm}
\usepackage{fancybox}
\usepackage[all, knot]{xy}

\newfont{\gl}{eufm10 scaled \magstep1} %% gothic fonts

\newcommand{\dd}{\hbox{d}}

\newcommand {\un}{\underline}
\newcommand {\cN}{\mathcal{N}}
\newcommand {\cM}{\mathcal{M}}

\newcommand {\cB}{\mathcal{B}}
\newcommand {\cI}{\mathcal{I}}
\newcommand {\cD}{\mathcal{D}}
\newcommand{\RR}{\ensuremath{\mathbb R}}
\newcommand{\CC}{\ensuremath{\mathbb C}}
\newcommand{\g}{\ensuremath{\frak{g}}}
\newcommand{\h}{\ensuremath{\frak{h}}}
\newcommand{\uu}{\ensuremath{\frak{u}}}
\newcommand{\cL}{\mathcal{L}}
\newcommand{\cO}{\mathcal{O}}

\numberwithin{equation}{section}
\newtheorem{theorem}{Theorem}[section]

\newtheorem{proposition}{Proposition}[section]

\newtheorem{lemma}{Lemma}[section]

\theoremstyle{remark}
\newtheorem{example}{Example}[section]

\theoremstyle{remark}
\newtheorem{remark}{Remark}[section]

\theoremstyle{definition}
\newtheorem{definition}{Definition}[section]

\begin{document}

\title{{An extension of the Marsden-Ratiu reduction for Poisson manifolds}}

\author{Fernando Falceto}
\address{Departamento de F\'{\i}sica Te\'orica and
Instituto de Biocomputaci\'on y F\'{\i}sica de Sistemas Complejos,
Universidad de Zaragoza,
E-50009 Zaragoza (Spain)}
\email{falceto@unizar.es}

\author{Marco Zambon}
\address{Centre de Recerca Matematica, Apartat de correus 50, 08193 Bellaterra
(Spain)}
\email{mzambon@crm.cat}


\begin{abstract}
\vskip 0.3cm
\noindent
We propose a generalization of the reduction of Poisson manifolds by 
distributions introduced by Marsden and Ratiu.
Our proposal overcomes some of the restrictions of the original procedure,
and  makes the reduced Poisson structure 
effectively dependent on the distribution. 
Different applications are discussed, as well as the algebraic 
interpretation of the procedure and its formulation in terms
of Dirac structures.
\end{abstract}
\maketitle
% \tableofcontents
%\vfill\eject

\section{Introduction}

Symplectic manifolds model phase spaces of physical systems, and
their theory of reduction   is a classical subject. A case in which
reduction occurs naturally is when a Lie group $G$ acts on a
symplectic manifold $(M,\omega)$ with equivariant moment map
$J:M\rightarrow \g^*$: under regularity assumptions the
Marsden-Weinstein theorem  states that the quotients
$J^{-1}(\mu)/G_{\mu}$ inherit a symplectic form. Another case is
given by submanifolds $C\subset M$ such that $TC^{\omega}\subset TC$
(coisotropic submanifolds), for in that case the quotient $C/
TC^{\omega}$, when smooth, inherits a symplectic form. The theory of
reduction  extends naturally to  Poisson manifolds, which encode
phase spaces of physical systems with symmetry.
% Recall
%that symplectic manifolds can be regarded as Poisson manifolds
%satisfying a non-degeneracy condition, and that Poisson manifolds .
The hamiltonian reduction and coisotropic reduction mentioned above
extend in a straightforward way to Poisson manifolds. Further,  both
are recovered as special cases of a reduction theorem stated in 1986
by Marsden and Ratiu  \cite{MarRat}.

The starting data of the Marsden-Ratiu theorem is a pair $(N,B)$
where $N$ is a submanifold  of the Poisson manifold $(M,\Pi)$ and
$B$  a subbundle of $T_NM$, the restriction of $TM$ to $N$. The
role of $B$ is to prescribe how
 to extend certain functions on $N$ to functions on the whole of
$M$, and is needed because the Poisson bracket of $M$ is defined
only for elements of $C^{\infty}(M)$. The conclusion of the theorem
is that, when the assumptions are met, the quotient $N/(B\cap TN)$
inherits a Poisson bracket from the one on  $M$.\\

The aim of this paper is two-fold.  First  we argue that the
assumptions of the Marsden-Ratiu theorem are too strong, in the
sense that the theorem allows to recover only  Poisson structures (on
quotients of $N$) which lose most of the information encoded by the
subbundle $B$. See Thm. \ref{MR2} -- our rephrasing of the original Marsden-Ratiu theorem -- and Prop. \ref{pr:unique}.
%For instance, in the
%simple case where $N$ is the plane $\{(0,y_1,0,y_2)\}$ in the
%standard $(\RR^4,dx_1\wedge dy_1+dx_2\wedge dy_2)$ and $B=span\{
%\frac{\partial}{\partial x_1}, \frac{\partial}{\partial x_2}-
%\lambda\frac{\partial}{\partial y_1}\}$, the assumptions of the
%Marsden-Ratiu theorem are satisfied only if $\lambda=0$. NOOO!!!

Then we set weaker assumptions on the pair $(N,B)$ which  ensure the
existence of  a Poisson structure on $N/(B\cap TN)$   encoding the
subbundle $B$. The main difficulty consists in ensuring that the
bracket of functions on the quotient satisfies the Jacobi identity.
 In
Prop. \ref{onn} we set assumptions  similar in spirit to those of
\cite{MarRat}, whereas in Prop. \ref{thetaD} the assumptions involve
an additional piece of data, namely a foliation on $M$. We apply
these results to the symplectic setting (with and without moment
map) as well.

The paper is organized as follows:
in the next section we review the original reduction of Marsden and Ratiu.
In section 3 we present the most general form of the extension 
that we propose, while section 4 is devoted to the application of the previous
results to some special situations and examples.
We collect in the appendix some complementary results,
like the algebraic interpretation of our reduction,
its description in term of Dirac structures and other auxiliary
material necessary for the main body of the paper. 

We finish remarking that an
extension of the Marsden-Ratiu reduction 
using supergeometry is
being worked out in \cite{CZsup}.\\

\noindent{\bf Acknowledgments:} We thank  J.P. Ortega  for pointing
out a necessary assumption in Prop. \ref{thetaD}.
 Research partially
supported by grants FPA2003-02948 and FPA2006-02315, MEC and
SB2006-0141(Spain) as well as SNF grant 20-113439 (Switzerland).

 \section{Marsden-Ratiu reduction}\label{MRpaper}

We start by recalling the Poisson reduction by distributions as it was stated
by Marsden and Ratiu in \cite{MarRat}, see also \cite{OrtRat}.
The set-up we consider here and in the rest of the paper is the following:
\begin{center}
\fbox{
\parbox[c]{12.6cm}{\begin{center}
$(M,\{\cdot,\cdot\})$ is a Poisson manifold\\
$N$ is a submanifold with embedding $\iota: N\hookrightarrow M$\\
${B}\subset T_N M$ is  a smooth
subbundle of $TM$ restricted to $N$.\\
${{F}}:={B}\cap TN$ is an integrable regular distribution on $N$.
\end{center}
}}
 \end{center}

In this section we shall also assume that $\un{N}:=N/F$ is a smooth manifold so that the projection map is a submersion.
Since we are concerned with  functions on $N$, instead of working with functions on the whole of $M$, we work on  a suitable tubular neighborhood  $M'$ of $N$ in $M$.  
% Passing from $M$ to a   tubular neighborhood $M'$ of $N$ in $M$ we can achieve that $\iota^* C^{\infty}(M')= C^{\infty}(N)$. To keep the notation simple in the sequel we will write $M$ instead of $M'$. } 
 
\begin{lemma}\label{surj}
%The map $$\iota^*: C^{\infty}(M)_B \rightarrow  C^{\infty}(N)_F$$ is
%surjective.  
There exists an open subset    $M'$ of $M$ containing $N$ with the property that  
any $F$-invariant function on $N$ can be extended to a $B$-invariant function on $M'$. 
\end{lemma}
\begin{proof} Fix a Riemannian metric on $M$.
Consider the normal\footnote{Here ``normal'' is taken to mean a  complement to $TN$ 
in $T_NM$, not necessarily the orthogonal bundle to $TN$ w.r.t. the Riemannian metric.}
bundle to $N\subset M$
given by $\pi: \tilde{B}\oplus (TN+B)^{\perp}  \rightarrow N$, where $\tilde{B}$ is the orthogonal complement to $F=B\cap TN$ in $B$, and $(TN+B)^{\perp}$ denotes the orthogonal complement to $TN+B=TN\oplus \tilde{B}$ in the vector bundle $T_NM$.
The exponential map associated to the
Riemannian metric on $M$ identifies a neighborhood of the zero section of this normal bundle
with a tubular neighborhood $M'$ of $N$ in $M$ \cite{MS}\cite{Lee}.  If $f\in C^{\infty}(N)_F$ then
$\pi^*f$ is an extension lying in  $C^{\infty}(M')_B$, for its differential at points of $N$ annihilates both $F$ and $\tilde{B}$.
\end{proof}

% with the property that  
%any $F$-invariant function on $N$ can be extended to a $B$-invariant function on $M'$. 
%The existence of such a tubular neighborhood $M'$ (in which $N$ lies as a a closed embedded submanifold of $M'$) is proven in Lemma \ref{surj}; 

The concrete choice of tubular neighborhood $M'$ is immaterial, because taking Poisson brackets is a local operation. To keep the notation simple, in the sequel we will assume that $M$ itself satisfies the above extension property, i.e. we will assume $M'=M$.

\begin{definition}\cite{MarRat}
The subbundle $B\subset T_N M$ is called
%{\bf Poisson}  or
{\bf canonical}
if for any elements $f_1, f_2$ of
$C^\infty(M)_{B} \equiv\{f\in C^\infty(M)\mid
\dd f\vert_{B}=0\}$ we have
$\{f_1,f_2\}\in C^\infty(M)_{B}$.
\end{definition}

In other words, $B$ is canonical if  the Poisson bracket of
${B}$-invariant functions is ${B}$-invariant.
Note that in the previous definition,
$\dd f\vert_{B}$ stands for the restriction (not pullback) of $\dd f$ to
$N$ and then to sections of ${B}$.
 
In the next definition we consider $C^\infty(N)_{{F}} \equiv\{f\in C^\infty(N)\mid
\dd f\vert_{F}=0\}$, which is naturally isomorphic to $C^\infty(\un{N})$.
\begin{definition}\label{def:reducible}\cite{MarRat}
%Let $(M,\{\cdot,\cdot\})$ be a Poisson manifold,
%$N$ a submanifold with embedding $\iota:N\hookrightarrow M$
%and ${B}\subset T_NM$ a subbundle with
 $(M,\{\cdot,\cdot\}, {N}, B)$ is {\bf Poisson reducible}
if there is a Poisson bracket $\{\cdot,\cdot\}_{\un{N}}$ on $\un{N}$
such that for any
$f_1, f_2\in  C^\infty(\un{N})\cong  C^\infty(N)_{{F}}$ we have:
$$\{f_1,f_2\}_{\un{N}}=\iota^*\{f_1^B,f_2^B\}$$
for all extensions $f_i^B\in C^\infty(M)_{B}$ of $f_i$.
\end{definition}

With the previous definitions we can state the Marsden-Ratiu reduction theorem.

%have the following

\begin{theorem}\label{th:MarRat}
(Marsden-Ratiu \cite{MarRat}) Assume that ${B}\subset T_NM$ is a
canonical subbundle. Then $(M,\{\cdot,\cdot\}, {N}, B)$  is Poisson
reducible if and only if
$$\sharp {B}^{\circ}\subset TN+{B}.$$

\end{theorem}
In the above theorem
$\sharp: T^*M \rightarrow TM$ denotes
the contraction with the Poisson bivector on $M$, and  
${B}^{\circ}={\rm Ann}({B})$ consists of elements of
$T_N^*M$ that kill all vectors in ${B}$.
The proof of the theorem can be found in \cite{MarRat} and
\cite{OrtRat}.\\

%The previous definitions and theorems are our starting point
%for this note.
 In the rest of this section we shall discuss the
implications of the  assumptions of the Marsden-Ratiu theorem.
% and in the remaining sections we will propose
%different generalizations and improvements.

The main observation is
%about the consequences of the assumptions of
%theorem \ref{th:MarRat}. In particular we will show
that the assumption made in Theorem \ref{th:MarRat}
that ${B}$ is canonical
 is a rather strong requirement.
 % statement as shown in
%the following proposition.

\begin{lemma}\label{condonMR}
%Let $(M,\{\cdot,\cdot\})$ be a Poisson manifold
%with Poisson tensor $\Pi\in\Lambda^2(TM)$,
%$N\subset M$ a submanifold and
Assume that ${B}\subset T_NM$ is a canonical subbundle.
% with ${B}\cap TN$ a regular, integrable distribution.
Then either $\sharp{B}^{\circ}\subset TN$ or ${B}=0$.
\end{lemma}
\begin{proof}
%The proof is pointwise.
Assume that there is a point $p\in N$
s. t. $(\sharp{B}^{\circ})_p\not\subset T_pN$. Then there is
a $B$-invariant function $h\in C^\infty(M)_{B}$
and a constraint
$g\in \cI\equiv\{f\in C^\infty(M)\ \mbox{s. t.}\ f\vert_N=0\}$
that satisfies $\{g,h\}(p)\not=0$.
%Consider now the previous function multiplied by itself.
It is clear that $g^2$ is ${B}$-invariant, and  the canonicity of
$B$ implies that $\dd\{g^2,h\}\vert_{B}=0$.
In particular one must have $i_{v}(\dd g)_p \{g,h\}(p)=0$
and we then deduce that $i_{v}(\dd g)_p=0$ for any $v\in{B}_p$.

Consider now any other constraint $g'\in \cI $,
we again have that $g\cdot g'$ is ${B}$-invariant
and therefore $i_{v}(\dd g')_p \{g,h\}(p)=0$.
From this we deduce that $i_{v}(\dd g')_p=0$ for any constrain $g'$ and any
$v\in{B}_p$. This is equivalent to saying
$${B}_p\subset T_pN.$$

%To go from the local, pointwise, condition to the global one it is enough
%to invoke
By the assumption of constant rank for ${B}\cap TN$
%. Hence if ${B}_p\subset T_p N$ at a point
we must have ${B}\subset TN$ everywhere.
This   implies that $f\cdot g$
is $B$-invariant
%\in C^\infty(M)_{B}$
for any $f\in C^\infty(M)$ and therefore
$i_{v}(\dd f)_p \{g,h\}(p)=0$ for any $v\in{B}_p$.
But this is possible only if ${B}_p=0$
which implies ${B}=0$ and the proof is complete.
\end{proof}

\begin{remark}
Consider the familiar situation in which $G$ is a compact Lie group 
acting freely on a symplectic manifold $(M,\omega)$ with equivariant 
moment map $J:M\rightarrow \g^*$.
Fix $\mu\in \g^*$, let $N=J^{-1}(\mu)$ and $B$ be given by the tangent 
spaces to the orbits of the $G$-action at points of $N$.
By Example B of \cite{MarRat} the subbundle $B$ is  canonical,  
and the Marsden-Ratiu theorem recovers the familiar symplectic 
structure on $J^{-1}(\mu)/G_{\mu}$.

Now    take $N$ as above but $B'\subset TN$ to be given by the
tangent spaces to the $G_{\mu}$-orbits  \emph{at points of $N$}, and
assume that $\mu$ is not a fixed point of the coadjoint action. Then
$B'$ is \emph{not} a canonical subbundle. This fact is consistent
with Lemma \ref{condonMR}, and is of course no contradiction to the
fact that the $G_{\mu}$-invariant functions on $M$ are closed under
the Poisson bracket (i.e. that the tangent spaces to the
$G_{\mu}$-orbits at \emph{all points of $M$} form a canonical
distribution).

% The fact  that $B'$ is \emph{not} a
%canonical subbundle is consistent with Lemma \ref{condonMR}.
%We show it directly as follows: let $e_1,e_2$  be two functions on
%$\g^*$ vanishing on $\mu$ such that $\{e_1,e_2\}_{\g^*}(\mu)=C\neq
%0$ (such functions exist since the coadjoint orbit through $\mu$ has
%dimension at least two). For any $f\in C^{\infty}(M)$ the functions
%$J^*e_1$ and $f\cdot J^*e_2$ vanish on $N$ and hence lie in
%$C^{\infty}(M)_{B'}$. Their Poisson bracket restricted to $N$ is
%$C\cdot f|_N$, and there is a choice of $f$ so that $f|_N$  is not
%$G_{\mu}$-invariant, i.e. so that $f\notin$ $C^{\infty}(M)_{B'}$.
\end{remark}

\begin{remark} If the subbundle $B \neq 0$ is canonical then it follows from Lemma \ref{condonMR}
 that $B^\circ \rightarrow N$ is a Lie subalgebroid of $T^*M$ (with the Lie algebroid structure induced by the Poisson structure on $M$).

 %Recall that   $T^*M \rightarrow M$ has a Lie algebroid
%structure  which   encodes the Poisson structure on $M$.
%The assumptions of Lemma \ref{condonMR} immediately imply that locally
%subbundle $B^\circ \rightarrow N$
%admits a frame of sections and extensions to $\Gamma(T^*M)$
%whose Lie algebroid brackets, restricted to $N$, are sections of $B$.
%Lemma \ref{condonMR} shows that, when $B\neq 0$, $B^\circ \rightarrow N$
%is a Lie subalgebroid of $T^*M$ (in particular that the choice of
%extensions is immaterial).
\end{remark}

In view of Lemma \ref{condonMR}, Theorem \ref{th:MarRat} 
becomes:

%\begin{proposition}\label{MR2}
%Assume that ${B}\subset T_NM$ a canonical subbundle and $B\not=0$.
%
%\end{proposition}

\begin{theorem}\label{MR2}
%\begin{center}
%\fbox{
%\parbox[c]{12.6cm}{\begin{center}
Assume that ${B}\subset T_NM$ is a canonical subbundle.
\begin{itemize}
\item  If $B\not=0$ then  $(M,\{\cdot,\cdot\}, {N}, B)$  is Poisson reducible.
\item If $B=0$, then $(M,\{\cdot,\cdot\}, {N}, B)$  is Poisson reducible iff $N$ is a Poisson submanifold.
\end{itemize}
%\end{center}
%}}
% \end{center}
\end{theorem}
\begin{proof}
If $B\not=0$ then by Lemma \ref{condonMR} we have $\sharp B^{\circ}\subset TN$, which
by Thm. \ref{th:MarRat} implies Poisson reducibility. For the case $B=0$ we can apply directly
Thm. \ref{th:MarRat}.
\end{proof}

Further, the induced Poisson structure on $\un{N}$ depends only on $F=B\cap TN$ and not on $B$, as stated by the following proposition.

\begin{proposition}\label{pr:unique}
If $B, B' \subset T_NM$ are non-zero canonical subbundles such that $B\cap TN=B'\cap TN$,
then the induced Poisson structures on $\un{N}=\un{N'}$ agree.
\end{proposition}
\begin{proof}
%Let $B, B'\subset T_NM$ with
By Lemma \ref{condonMR} we know that $\sharp B^{\circ},\sharp B'^{\circ}\subset TN$.
%This implies that  for any
Let $f_1,f_2\in C^\infty(M)_B$ and
$f_1', f_2'\in C^\infty(M)_{B'}$ s. t.
$\iota^*f_i=\iota^*f'_i$, $i=1,2$.
%here $\iota:N\rightarrow M$ is the embedding
%of the submanifold $N$.
We have
$$\iota^*\{f_1,f_2\}-\iota^*\{f_1',f_2'\}
=\iota^*\{f_1-f_1',f_2\}+\iota^*\{f_1',f_2-f_2'\},$$
and as for any $p\in N$ we have $d(f_i-f_i')_p\in TN^\circ$ and
$\sharp(df_i)_p\in TN$, the right hand side vanishes. Hence
$$\iota^*\{f_1,f_2\}=\iota^*\{f_1',f_2'\}.$$
\end{proof}

\begin{remark} To every submanifold $N$ of the Poisson manifold $M$
  is   canonically associated a Poisson algebra, as
follows\footnote{This is an algebraic version of Example $D$ in
 \cite{MarRat}; the latter holds when $\sharp TN^{\circ}$ and $\sharp TN^{\circ}\cap TN$
  have constant rank.}. Let  $\cI$ be the ideal of functions on
  $M$ vanishing on $N$. Its Poisson normalizer $\cN\equiv\{f\in C^\infty(M)\mid
\{f,\cI\}\subset \cI\}$ is a Poisson subalgebra, so the
quotient $\cN/(\cN\cap \cI)$ is a Poisson algebra (see also \cite{glmv}).
Notice that $\cN$ consists of functions whose differentials
annihilates all vectors in  $\sharp TN^{\circ}$.

Now let $B$ be a nonzero canonical subbundle. Then
$C^{\infty}(\un{N})$, with the Poisson bracket induced as in Thm.
\ref{MR2}, is a \emph{Poisson subalgebra} of  $\cN/(\cN \cap \cI)$. Indeed by Lemma \ref{condonMR} we have $B \supset
\sharp TN^{\circ}$, so $C^{\infty}(M)_B\subset \cN$, hence
$C^{\infty}(\un{N})=C^{\infty}(M)_B/(C^{\infty}(M)_B\cap \cI)$
sits inside $\cN/(\cN\cap \cI)$ and is a Poisson
subalgebra. Notice that $\cN/(\cN\cap \cI)$ does not
``see'' the subbundle $B$, in agreement with Prop. \ref{pr:unique}
above. \end{remark}

%\begin{remark} We complete Prop. \ref{MR2} and Prop. \ref{pr:unique} by dealing with  
%the trivial case $B=0$ (which is clearly canonical).
%$(M,\{\cdot,\cdot\}, {N}, B=0)$ is Poisson reducible iff $N$ is a Poisson submanifold.
%If $B'$ is some canonical subbundle with $B'\cap TN=0$ then the Poisson structures %induced by $B'$ and $B=0$ on $N$ agree, as $N$ is a Poisson submanifold.
%\end{remark}

\begin{remark} We complete 
Prop. \ref{pr:unique} by dealing with  the trivial case $B=0$ (which is clearly canonical). As we saw above,
$(M,\{\cdot,\cdot\}, {N}, B=0)$ is Poisson reducible iff $N$ is a Poisson submanifold.
If $B'$ is some canonical subbundle with $B'\cap TN=0$ then the Poisson structures induced by $B'$ and $B=0$ on $N$ agree, as $N$ is a Poisson submanifold.
\end{remark}

The conclusion of Prop. \ref{pr:unique} is that, when the Marsden-Ratiu reduction endows
 $\un{N}$ with an
induced Poisson structure, this structure depends only on $F$.
%(use Prop. \ref{MR2} and Lemma \ref{pr:unique} with $B\cap TN=B'\cap TN=F$).
%Due to Lemma \ref{pr:unique},
%Then if $B\cap TN=B'\cap TN:=F$ the Poisson structures
%induced on $N/F$ by the distributions $B$ and $B'$ are the same.
This result is against the original idea of reduction by distributions,
where the role played by $B$ is expected to be more prominent.
In order to accomplish this objective we will proceed, in the coming section,
to relax the condition of canonicity of the distribution while maintaining
the requirement of having a Poisson structure induced on  $\un{N}$.

\section{Extension of the Marsden-Ratiu reduction}\label{extMR}

The set-up of this section consists of  the geometric data of the
Mardsen-Ratiu theorem; we will set various conditions on these data
which guarantee Poisson reducibility.
So   let $(M,\Pi)$ be a Poisson manifold,
$N\subset M$ a submanifold and $B \subset T_NM$ a   subbundle
with $F:=B\cap TN$ a regular, integrable distribution.  Without loss of generality (see
Lemma \ref{surj}), here and in the rest of the paper, we 
 assume that the restriction map
$\iota^*: C^{\infty}(M)_B \rightarrow  C^{\infty}(N)_F$ is
surjective.
We do not need to assume that
$\un{N}:=N/F$
is a smooth manifold, even though this is of course the case of interest. In that case $C^{\infty}(\un{N})\cong
C^{\infty}(N)_F$.

We would like to  define a bilinear operation $\{\cdot,\cdot\}_{\un{N}}$ on $C^{\infty}(N)_F$ by the following rule:
 \begin{equation}\label{defbr}
\{f,g\}_{\un{N}}:=\iota^*\{f^B,g^B\}
\end{equation}
where  $f^B$, $g^B$ are arbitrary extensions to elements of
$C^{\infty}(M)_B$.  
 As the restriction map
$\iota^*: C^{\infty}(M)_B \rightarrow  C^{\infty}(N)_F$ is
surjective there is at most one bilinear
operation $\{\cdot,\cdot\}_{\un{N}}$. Our task is to determine when
$\{\cdot,\cdot\}_{\un{N}}$ is well-defined and when it is a Poisson
bracket.\\

The r.h.s. of eq. \eqref{defbr} is independent of the chosen extensions
(for all $f,g \in  C^{\infty}(N)_F$) iff
\begin{equation}\label{tnb}\sharp B^{\circ}\subset
TN+B \end{equation}
 (see the proof of the Marsden-Ratiu theorem or the proof of Thm.
\ref{gen} below). If this is the case, the r.h.s. of \eqref{defbr}
lies in  $C^{\infty}(N)_F$ iff for one choice of  extensions
$f^B,g^B$ we have $\iota^*\{f^B,g^B\} \in C^{\infty}(N)_F$, or
equivalently if
\begin{equation}\label{desc} \{ {C^{\infty}(M)_B},
{C^{\infty}(M)_B}\}\subset C^{\infty}(M)_F. \end{equation}

In this case clearly $\{\cdot,\cdot\}_{\un{N}}$ will be a
skew-symmetric operation on $C^{\infty}(N)_F$ which is a
biderivation w.r.t. the product; if $\un{N}$ is smooth, this means
that $\{\cdot,\cdot\}_{\un{N}}$ defines a bivector field on it.

%\begin{remark}\label{rembsb}
%Conditions \eqref{tnb} and \eqref{desc} imply a relation that is
%easy to check in practice, namely that  ${\rm rk}(B+\sharp
%B^{\circ})$ is constant along each leaf of the distribution $F$.
%This will be proven in the Appendix. See Ex. \ref{exr6} for an
%example where this condition is violated.
%\end{remark}

Now we want to determine conditions under which
$\{\cdot,\cdot\}_{\un{N}}$ satisfies the Jacobi identity, for when this is the case
$(M,\{\cdot,\cdot\}, {N}, B)$ is Poisson reducible.
 Checking
the Jacobi identity suggests to require that  for any $f,g \in
C^{\infty}(N)_F$ there \emph{exist} extensions $f^B,g^B$ whose bracket
annihilate not only $F$ but actually a larger subbundle  (not necessarily tangent to $N$). This leads us to  a condition that
  involves two pieces of data: an additional subbundle $D$ of
$T_NM$ and a subspace $\cB$ of $C^{\infty}(M)_B$ which contains the above
extensions.
%Later on we will interpret $D$ as an object by  which we
%want to \emph{divide} the manifold $M$ and the subspace of
%$C^{\infty}(M)_B$ as distinguished functions related to the quotient
%space.
In the Appendix we give an algebraic interpretation of these data,
and at the end of Subsection \ref{appldis} we give a geometric
interpretation.

 \begin{theorem}\label{gen}
Let $(M,\{\cdot,\cdot\})$ be a Poisson manifold, $N \subset M$ a submanifold and
$B \subset T_NM$ a   subbundle with $F:=B\cap TN$ a regular,
integrable distribution.
Let $D$ be a subbundle of $T_NM$ satisfying\footnote{Equivalently
$D\subset B$ and $B\cap TN=D\cap TN$.} $F\subset D \subset B$ and
\begin{equation}\label{sha}
\sharp B^{\circ}\subset D+TN.
\end{equation}
 Let $\cB \subset C^{\infty}(M)_B$ be a
multiplicative subalgebra such that the restriction map $\iota^*:
\cB  \rightarrow  C^{\infty}(N)_F$
%\cong C^{\infty}(\un{N})$
is surjective. Assume that
\begin{equation}\label{bra}
\{\cB ,\cB \}\subset
C^{\infty}(M)_D.\end{equation}

Then $(M,\{\cdot,\cdot\}, {N}, B)$ is Poisson reducible.
\end{theorem}

\begin{proof}
Consider functions $f,g \in C^{\infty}(N)_F$ and extensions
$f^B,g^B$ in $\cB $. If we choose a different
extension ${f^B}'$ for $f$, the differential of  $f^B-{f^B}'$
annihilates $TN+B$, so because of $\sharp(TN+B)^{\circ}\subset
\sharp(D+TN)^{\circ}\subset B$ (eq. \eqref{sha}) we have $\iota^* \{f^B-{f^B}',g^B\}=0$.
Hence the expression for $\{f,g\}_{\un{N}}$ is independent of the
choice of extensions. By eq. \eqref{bra} it actually lies in
$C^{\infty}(N)_F$.

Now $\{f^B,g^B\}$ and $(\{f,g\}_{\un{N}})^B$ by definition agree on
$N$, and are elements respectively of $C^{\infty}(M)_D$ (by eq.
\eqref{bra}) and $\cB $. So their difference
annihilates $D+TN$ and by eq. \eqref{sha} the Poisson bracket of their difference with any
element of $\cB $ vanishes on $N$. This explains the second equality in the identity
$$\{\{f,g\}_{\un{N}},h\}_{\un{N}}=\iota^* \{(\{f,g\}_{\un{N}})^B,h^B\}=
\iota^* \{\{f^B,g^B\},h^B\}.$$ From this is clear that the Jacobi
identity for $\{\cdot,\cdot\}_{\un{N}}$ holds as a consequence of
that for $\{\cdot,\cdot\}$.
\end{proof}

\begin{remark} Enlarging $D$ makes the constraint  \eqref{bra} more
severe, so in applications one should choose $D$ satisfying
\eqref{sha} to have dimension as small as possible. In general there
is no  unique minimal choice of $D$.
\end{remark}

\section{Applications and examples} In this section we consider special cases
of Thm. \ref{gen}. As usual  $(M,\Pi)$ is a Poisson
manifold, $N\subset M$ a submanifold and $B \subset T_NM$ a
subbundle with $F:=B\cap TN$ a regular, integrable distribution.
% so that $\un{N}:=N/F$ is a smooth manifold.

\subsection{A straightforward application}
 Setting $D=F$ and
$\cB =C^{\infty}(M)_B$ in Thm. \ref{gen}. we
obtain a minor improvement of the Marsden-Ratiu theorem (Thm. \ref{th:MarRat}), where the
condition on the canonicity of $B$ is weakened:
\begin{proposition}\label{onn}
If
\begin{equation} \label{BBF} \{C^{\infty}(M)_B,C^{\infty}(M)_B\}\subset C^{\infty}(M)_F
\end{equation} and $\sharp B^{\circ}\subset TN$ then
$(M,\{\cdot,\cdot\}, {N}, B)$ is Poisson reducible.

\end{proposition}

\begin{remark}\label{remliederF}
In the above proposition condition \eqref{BBF} is equivalent to
the following, which is more suited for computations: locally
  there exists a frame of sections $X_i$ of $F$ and extensions thereof to
  vector fields on $M$ such that
 \begin{equation}\label{liederF} (\cL_{X_i} \Pi)|_N \subset B \wedge
 T_NM.\end{equation}
This can be shown  using formula \eqref{Xfg} below and  $\sharp
B^{\circ}\subset TN$.
\end{remark}

We present an example where the assumptions of Prop. \ref{BBF} are
satisfied but $B$ is not canonical.
\begin{example}\label{zxy}
Let $(M,\Pi)$ be $(\RR^3, z\frac{\partial}{\partial x}\wedge
\frac{\partial}{\partial y})$ and $N$ the plane given by $z=0$. Let
$B=\RR \frac{\partial}{\partial z}$. The conditions of Prop.
\ref{BBF} are satisfied because $\Pi$ vanishes at points of $N$ and
because $F=B\cap TN=\{0\}$. However $C^{\infty}(M)_B$ is not closed
w.r.t. the Poisson bracket: for instance $x,y$ lie in
$C^{\infty}(M)_B$ but $\{x,y\}=z$ does not.
\end{example}

If $\sharp B^{\circ}\subset TN$ then necessarily $\sharp
TN^{\circ}\cap TN \subset F$. When this last inclusion is an
equality eq. \eqref{BBF} holds automatically, so the interesting
case is when the inclusion is strict, as in the following example,
in which we use Remark \ref{remliederF}  to check condition
\eqref{BBF}.
\begin{example}
Let $(M,\Pi)=(\RR^6, \sum_i \frac{\partial}{\partial x_i}\wedge
\frac{\partial}{\partial y_i})$ and $N$ be the (coisotropic)
hyperplane $\{y_3=0\}$.  Let $B={\rm span}\{\frac{\partial}{\partial x_3},
\frac{\partial}{\partial x_1}+ \alpha\frac{\partial}{\partial
y_2}\}\subset TN$ where $\alpha \in C^{\infty}(N)$. $F=B$ is integrable iff
$\alpha$ is independent of $x_3$. We have $\sharp B^{\circ}\subset
TN$  since $B$ contains the characteristic distribution of $N$.

We check  condition \eqref{liederF}, which is easier than checking
directly  condition \eqref{BBF}. We have $\cL_{\frac{\partial}{\partial x_3}}\Pi=0$,
and $(\cL_{ \frac{\partial}{\partial x_1}+
\alpha\frac{\partial}{\partial y_2}}\Pi)|_N=-\sharp \dd\alpha\wedge
\frac{\partial}{\partial y_2}$ surely lies in $B \wedge T_NM$ if
$\alpha$ depends only on the coordinates $y_1$ and $x_2$. In this
case by Prop.  \ref{onn} the quotient $\un{N}\cong \RR^3$ has an
induced Poisson structure, which in suitable coordinates is given by
$\{y_1,y_2\}=\alpha$, $\{x_2,y_2\}=1$ and $\{y_1,x_2\}=0$.
\end{example}

\subsection{An application involving distributions}\label{appldis}

 If $M$ is endowed with a suitable distribution we can
weaken the condition $\sharp B^{\circ}\subset TN$ (which, as seen in
Lemma \ref{condonMR}, is an assumption of the Marsden-Ratiu theorem
for $B\not=0)$.
%We state the result in global terms, but see Remark
%\ref{remlocal} below.

\begin{definition}\label{compat}
Let  $\theta_D$ be an integrable distribution  on $M$ such that
$F\subset {\theta_D}|_N \subset B$. We say that  $\theta_D$ and
$B$ are {\bf compatible} if  $\iota^*:C^{\infty}(M)_B\cap
C^{\infty}(M)_{\theta_D}\rightarrow C^{\infty}(N)_F$ is surjective.
\end{definition}

The above compatibility is satisfied for instance when ${\theta_D}|_N = B$ or
$F:=B\cap TN=\{0\}$.
In the appendix (Prop. \ref{equichar})
we shall give an equivalent characterization of
Def. \ref{compat}.

\begin{proposition}\label{thetaD}
Suppose that on $M$  there is an integrable distribution $\theta_D$
 such that $F\subset D:={\theta_D}|_N \subset B$ and so that
 $\theta_D$ is compatible with $B$.
 Assume that
 \begin{equation}\label{shathetaD}\sharp B^{\circ}\subset
 D+TN\end{equation} and that, for any section $X$ of $\theta_D$,
 \begin{equation}\label{lieder} (\cL_X \Pi)|_N \subset B \wedge T_NM.
 \end{equation}
 Then $(M,\{\cdot,\cdot\}, {N}, B)$ is Poisson reducible.
\end{proposition}
\begin{proof} Set
$\cB =C^{\infty}(M)_B\cap C^{\infty}(M)_{\theta_D}$ in Thm.
\ref{gen}. By assumption  $\iota^*:\cB\rightarrow
C^{\infty}(N)_F$ is surjective.
 Condition \eqref{bra}  reads $$
%\label{brathetaD}
\{C^{\infty}(M)_B \cap
C^{\infty}(M)_{\theta_D},
 C^{\infty}(M)_B \cap C^{\infty}(M)_{\theta_D}\}\subset
C^{\infty}(M)_D.$$ This is equivalent to \eqref{lieder}, as one can see
evaluating at points of $N$ the following equation: for $X \in
\Gamma(\theta_D)$ and $f,g\in C^{\infty}(M)_B \cap
C^{\infty}(M)_{\theta_D}$,
\begin{equation}\label{Xfg} X\{f,g\}=
%\cL_X(\Pi(df,dg))=
(\cL_X
\Pi)(df,dg)+\Pi(d(Xf),dg)+\Pi(df,d(Xg)).\end{equation}
 \end{proof}

\begin{remark}\label{remlocal}
 It is sufficient to apply Prop. \ref{thetaD} locally.
More precisely: let $\{U^{\alpha}\}$ be an open cover
of a tubular neighborhood of $N$, and suppose that on each
$\{U^{\alpha}\}$
there exists an integrable distribution $\theta^{\alpha}_D$
as in the proposition. Then in particular eq. \eqref{tnb} is
satisfied, so eq. \eqref{defbr} determines a well-defined map
$C^{\infty}(N)_F\times  C^{\infty}(N)_F \rightarrow  C^{\infty}(N)$.
Applying Prop. \ref{thetaD} on each open set $U^{\alpha}$ ensures that this map
defines a Poisson bracket on $C^{\infty}(N)_F$.

Further it is sufficient to check condition \eqref{lieder}
locally on a frame $\{X_i\}$ of sections of $\theta_D$.
\end{remark}

To further illustrate the properties of the reduction discussed in
Prop. \ref{thetaD} we provide some concrete examples that highlight
different aspects of the reduction. The first examples are
particularly simple, since  there $B\oplus TN=T_NM$, so that
formula \eqref{defbr} defines a bivector field on $\un{N}=N$.
\begin{example}\label{poissonex}
Consider the symplectic manifold $(\RR^4, \sum_i dx_i\wedge dy_i)$,
 let $N$ be given by the constraints $x_1=x_2=0$ and let
$B={\theta_D}|_N$ where $\theta_D$ is the distribution on $\RR^4$
given by
  $\theta_D:=span\{\frac{\partial}{\partial x_1},
\frac{\partial}{\partial x_2}- \lambda\frac{\partial}{\partial
y_1}\}$ (with $\lambda \in \RR$).  $C^{\infty}(M)_{\theta_D}$ is
closed under the bracket and $\sharp (TN+B)^{\circ}=0$, so the
assumptions of
 Prop. \ref{thetaD} are met. The quotient $\un{N}$ is $\RR^2$ with
 natural coordinates $y_1,y_2$ and Poisson bivector
 $\lambda\frac{\partial}{\partial y_1}\wedge \frac{\partial}{\partial
 y_2}$.

 Notice that in this example the condition $\sharp B^{\circ}\subset
 TN$ is violated. The final Poisson structure depends on $B$
(while $B\cap TN=\{0\}$ is independent of $\lambda$). As shown in Prop.
\ref{pr:unique} this  can not happen in the Marsden-Ratiu reduction
(Thm. \ref{th:MarRat}).
\end{example}

%A slight variation of Ex. \ref{poissonex} shows that submanifolds
%which have a canonically induced Poisson structure, such as
%symplectic submanifolds, can be endowed with Poisson structures
%different from the canonical one.
%\begin{example}\label{exsympl}
%Let $M$ be  $(\RR^4, \sum_i dx_i\wedge dy_i)$.  Let $N$ be given by
%the constraints $x_2=y_2=0$. Let $B={\theta_D}|_N$ where $\theta_D$
%is
%  the distribution
%$\theta_D:=span\{\lambda \frac{\partial}{\partial
%x_1}+\frac{\partial}{\partial x_2}, \nu \frac{\partial}{\partial
%y_1}+  \frac{\partial}{\partial y_2}\}$ (where $\lambda, \nu \in
%\RR$).
% $C^{\infty}(M)_{\theta}$ is closed under the bracket and
%$\sharp (TN+B)^{\circ}=0$, so
%As in the previous example the assumptions of
% Prop. \ref{thetaD} are met. The quotient $\un{N}$ is $\RR^2$ with
% natural coordinates $x_1,y_1$ and Poisson bivector
% $(1+\lambda \nu)\frac{\partial}{\partial x_1}\wedge \frac{\partial}{\partial
% y_1}$.
% Notice that in this example the condition $\sharp (TN)^{\circ}\subset
% B$ is violated and the final Poisson structure depends on the distribution
% while $B\cap TN=0$ is independent of $\lambda$.
%\end{example}

The next example illustrates the fact that,
even if we have a well-defined smooth bivector on $\un{N}$,
we need extra conditions to satisfy the Jacobi identity.
\begin{example}\label{exbsb}
 Let $(M,\Pi)$ be the Poisson manifold
$(\RR^4, \sum_i \frac{\partial}{\partial x_i}\wedge
\frac{\partial}{\partial y_i})$, consider the hyperplane
$N=\{y_2=0\}$ and the subbundle $B$ of $T_NM$ spanned by
$\frac{\partial}{\partial
y_2}+ \alpha \frac{\partial}{\partial x_1}$,
 where $\alpha
\in C^{\infty}(N)$.
%We have
%$\sharp B^{\circ} \subset TN+B$ and
%$F:=B \cap TN=\{0\}$,
%so formula \eqref{defbr} defines a bivector field on $N$.
The bivector field induced by eq. \eqref{defbr} on $N$ is
$\frac{\partial}{\partial x_1}\wedge (\frac{\partial}{\partial
y_1}+\alpha \frac{\partial}{\partial x_2})$, hence it is Poisson iff
$\alpha$ is independent of $x_1$.

All the Poisson structures obtained above can be obtained using
Prop. \ref{thetaD}. Indeed, if we extend $B$ to the distribution
$\theta_D:=\RR ( \frac{\partial}{\partial y_2}+ \alpha
\frac{\partial}{\partial x_1})$, eq. \eqref{lieder} is satisfied iff
$\frac{\partial}{\partial x_1}\alpha=0$.
%, i.e. all the Poisson structures above are recovered.
%
%Notice that if $\alpha$ is independent of $x_1 $ and we extend it to
%$\alpha^M \in C^{\infty}(M)$ with
%$\frac{\partial}{\partial y_2}\alpha^M$ vanishing on $N$ but not on
%the whole of $M$, the relation
%$\cL_{X} \Pi\subset \theta_D \wedge TM$ for
%$X=\frac{\partial}{\partial y_2}+ \alpha^M \frac{\partial}{\partial x_1}$
%is not satisfied, but  eq. \ref{lieder} is.
\end{example}

In the previous example we have seen an obstruction to obtaining a
Poisson structure after the reduction, namely eq. \eqref{lieder}. In
the following example the distribution $F$ on $N$ is non-trivial,
and we shall also exhibit an obstruction to have a well defined
bivector field on $\un N$ in the first place.

\begin{example}\label{exr6} Let $(M,\Pi)$ be the Poisson manifold
$(\RR^6, \sum_i \frac{\partial}{\partial x_i}\wedge
\frac{\partial}{\partial y_i})$, consider the hyperplane
$N=\{y_2=0\}$ and the subbundle of $T_NM$ given by
$B={\rm span}\{\frac{\partial}{\partial y_1}, \frac{\partial}{\partial
y_2}+ \alpha \frac{\partial}{\partial x_1}\}$, where
$\alpha\in C^{\infty}(N)$. Clearly $\sharp B^{\circ} \subset TN+B$, and
$F:=B \cap TN=\RR \frac{\partial}{\partial y_1}$.

Now the bracket of the $B$-invariant extensions
of the coordinate functions $x_1,x_2$ is:
$\{x_1^B, x_2^B\}\vert_N=\alpha$, which is well defined on $\un N$
iff $\alpha$ does not depend on $y_1$.  This condition ensures that we have
a bivector field on $\un N$ but still is not enough
to guarantee reducibility.
%We compute the rank of $B+\sharp B^{\circ}$: it is $4$ at points of
%$N$ where $\alpha=0$ (because $B$ is an isotropic subspace there)
%and $6$ elsewhere. Hence, by Remark \ref{rembsb}, a necessary
%condition for $\Pi$   to induce the almost Poisson structure
%\eqref{defbr} on $\un{N}$ is the following: if $\alpha$ vanishes at
%a point $m$ then it has to vanish on the whole leaf of $F$ through
%$m$.

Prop. \ref{thetaD} can be applied to determine when the bracket
$\{\cdot, \cdot\}_{\un{N}}$ is a well-defined Poisson bracket. We
%choose $D:=B$ (this is the only possible choice) and
extend $B$ constantly in the $y_2$ direction to obtain the
distribution $\theta_D=span\{\frac{\partial}{\partial y_1},
\frac{\partial}{\partial y_2}+ \alpha\frac{\partial}{\partial
x_1}\}$ on $M$.  The distribution $\theta_D$ is integrable iff
$\alpha$ does not depend on $y_1$.  Now
$\cL_{\frac{\partial}{\partial y_1}}\Pi=0$, and
$\cL_{\frac{\partial}{\partial y_2}+ \alpha\frac{\partial}{\partial
x_1}}\Pi= -X_{\alpha}\wedge \frac{\partial}{\partial x_1} \subset
B\wedge T_NM$
iff $\alpha$ does not depend on the coordinates $x_3$ and $y_3$.
Hence Prop. \ref{thetaD} allows us to conclude that, when $\alpha$
depends only on the coordinates $x_1$ and $x_2$, we obtain a Poisson
bivector on $\un{N}\cong \RR^4$. In the natural coordinates, the
induced Poisson bivector is $\alpha \frac{\partial}{\partial
x_1}\wedge \frac{\partial}{\partial x_2} +\frac{\partial}{\partial
x_3}\wedge \frac{\partial}{\partial y_3}$.
\end{example}

The following is a simple example in which $M$ is a linear Poisson
manifold.
\begin{example}\label{dual}
Let $\g$ be a Lie algebra, $V\subset \g$ a subspace and $\h\subset
\g$ a Lie subalgebra satisfying $[\h, V\cap \h]\subset V$. We set
$M:=\g^*$, $N:=V^{\circ}$, and $B_x:=\h^{\circ}\subset T_xM$ at all
$x\in N$. Using  Lemma 5.4 of \cite{CZbis} and the assumptions,  we
see $\sharp B_x^{\circ}=\{ad_h^*(x): h\in \h\}\subset T_xN+B_x$ at
all $x\in N$. Extending $B=\h^{\circ}$ by translation to a
distribution $\theta_D$ on $M$ and noticing that the projection
$\g^*\rightarrow \g^*/\h^{\circ}\cong \h^*$ is a Poisson map we see
that eq. \eqref{lieder} is satisfied. By Prop. \ref{thetaD} we
conclude that there is an induced (linear) Poisson structure on
$\un{N}=\frac{V^{\circ}}{V^{\circ}\cap \h^{\circ}}\cong
(\frac{V+\h}{V})^*$. It corresponds to the Lie algebra structure on
$\frac{\h}{\h\cap V}$, which as a vector space is canonically
isomorphic to $\frac{V+\h}{V}$.
\end{example}

Our last example  shows that conditions of Prop. \ref{thetaD} are
not necessary
%. As  we show in the example below, the asssumptions
%of the proposition can not be met and we still have
in order to obtain a Poisson structure after the reduction.
\begin{example}
 Let $(M,\Pi)$ be $(\RR^3, z\frac{\partial}{\partial x}\wedge
\frac{\partial}{\partial y})$,  $N$ the plane given by $z-x=0$ and
$B=\RR \frac{\partial}{\partial z}$. Formula \eqref{defbr} defines
the Poisson structure $\{x,y\}=x$ on $N$, however
  Prop.
\ref{thetaD} can not be applied   because a distribution $\theta_D$
as in the proposition does not exist. Indeed $\theta_D$ has to be
one-dimensional because of eq. \eqref{shathetaD}.
  For any vector field
$X$ which restricts to $\frac{\partial}{\partial z}$ on $N$, we have
$(\cL_X \Pi)|_p=X(z)|_p \cdot \frac{\partial}{\partial x}\wedge
\frac{\partial}{\partial y}=\frac{\partial}{\partial x}\wedge
\frac{\partial}{\partial y}$ at any point   $p\in N$ of the form
$(0,y,0)$, so eq. \eqref{lieder} is not satisfied.
\end{example}

% \begin{example}
% If one takes $M=\RR^3$, $N=\{z=0\}$, $B=span\{
% \frac{\partial}{\partial x}, \frac{\partial}{\partial z}+x
% \frac{\partial}{\partial y}\}$ and $\theta_D=\RR
% \frac{\partial}{\partial x}$, only constant functions on $N$ admit a
% $\theta_D$-invariant extension which annihilates $B$.
% Notice that $\nabla_{\frac{\partial}{\partial x}}
% (\frac{\partial}{\partial z}+x
% \frac{\partial}{\partial y} \text{ mod }\theta_D)=
% \frac{\partial}{\partial y} \notin F$,
% violating eq. \eqref{higher}.
% A distribution compatible with $B$ is given by
% $\theta_D=\RR
% (\frac{\partial}{\partial x}+z
% \frac{\partial}{\partial y})$.\end{example}

We conclude the subsection giving a geometric interpretation
of Prop. \ref{thetaD}. Assume
that the quotient $\un{M}:=M/\theta_D$ is smooth  and that
$C^{\infty}(M)_{\theta_D}$ is closed under the Poisson bracket, so
that $\un{M}$ has a Poisson structure  for which the projection $M
\rightarrow \un{M}$ is a Poisson map.  Assume $\un{N} \subset
\un{M}$ is a { Poisson-Dirac submanifold} \cite{CrF}, so that it has
an induced Poisson structure. Then the Poisson bracket of functions
on $\un{N}$ is computed by lifting to functions in
$C^{\infty}(M)_{B}$ where $B$ is  a subbundle   as in Prop.
\ref{thetaD}. In this interpretation the case $D=B$ corresponds to
the case where $\un{N}$ is actually a
\emph{Poisson} submanifold of $\un{M}$.\\

%We give a geometric interpretation of Prop. \ref{thetaD}. Assume
%that the distribution $\theta_D$ is integrable with smooth quotient
%$\un{M}:=M/\theta_D$   and that $C^{\infty}(M)_{\theta_D}$ is closed
%under the Poisson bracket, so that $\un{M}$ has a Poisson structure
%$\un{\Pi}$   for which the projection $p:M \rightarrow \un{M}$ is a
%Poisson map. Assume that $\un{N}$ is an embedded submanifold of
%$\un{M}$,
%% (a priori  it is just immersed).
%and for the sake of simplicity assume that
%$\un{\sharp}T\un{N}^{\circ}$ has constant rank, where $\un{\sharp}$
%denotes contraction with the Poisson bivector field on $\un{M}$.
% A simple computation shows
%that the preimage of $\un{\sharp}T\un{N}^{\circ}$ under $p_*$ is
%$B:=D+\sharp (D+TN)^{\circ}$. The assumptions of Prop. \ref{thetaD}
%are satisfied if
% $B\cap TN=D\cap TN$, which is equivalent to
 %$\un{\sharp}T\un{N}^{\circ}\cap T\un{N}=\{0\}$. This last condition
 %implies that $\un{N}$ is a \emph{ Poisson-Dirac submanifold}
%\cite{CrF}, so
%Being a Poisson-Dirac submanifold
%$\un{N}$ has an induced Poisson structure, obtained extending
%functions on $\un{N}$ to elements of
%$C^{\infty}(\un{M})_{\un{\sharp}T\un{N}^{\circ}}$. In terms of $M$
%and $N$, using the fact that $p$ is a Poisson map, this means that
%$C^{\infty}(N)_F$ is endowed  with a Poisson bracket obtained
%extending to functions on $M$ which descend to $\un{M}$ and
%annihilate $(p_*)^{-1} \un{\sharp}T\un{N}^{\circ}= B$. This is
%exactly the Poisson structure on $\un{N}$ induced by Prop.

\subsection{An application to hamiltonian actions}

Here is an instance where the assumptions of Prop. \ref{thetaD} are
naturally met. Given an action of a Lie group on a manifold $M$ we
denote by $\g_M(p)$ the span at $p\in M$ of the vector fields
generating the action (i.e. the tangent space of the $G$-orbit
through $p$).
\begin{proposition}\label{action}
Let the Lie group $G$ act on the symplectic manifold $(M,\omega)$ so
that $\g_M$ has constant rank and  with equivariant moment map
$J:M\rightarrow \g^*$. Let $m\in J^{-1}(0)$ and $N$ be a submanifold through 
$m$ so that
\begin{equation}\label{dirsum}
T_mN\oplus ker(d_mJ)=T_mM.
\end{equation}
Then $N$,  after shrinking it to a smaller neighborhood of $m$ if
necessary, has an induced Poisson structure, obtained extending
functions from $N$ to $M$ so that they annihilate
$[\g_M+(TN+\g_M)^{\omega}]|_N$
 %elements of
%$C^{\infty}(M)_{\g_M+(TN+\g_M)^{\omega}}$.
\end{proposition}
\begin{proof} Consider $B:=[\g_M+(TN+\g_M)^{\omega}]|_N\subset T_NM$ and the
distribution $\theta_D:=\g_M$. We now check that the assumptions of
Prop. \ref{thetaD} are automatically satisfied; we will make use
repeatedly of $\g_M(m)\subset ker(d_mJ)=\g_M(m)^\omega$, which
holds by the equivariance of $J$.

First of all $B$ has constant rank, at least near $m$. Indeed the
sum of $TN$ and $\g_M$ has constant rank because their intersection
at $m$ is trivial. Further
 $TN+\g_M$ is a symplectic subbundle of $T_NM$.
 To this aim we check that at the point $m$ we have
\begin{equation}
\begin{aligned}
T_mN^{\omega}\cap \left[\g_M(m)^{\omega}\cap (T_mN+\g_M(m))\right]&=
T_mN^{\omega}\cap
\g_M(m)=\hfill\cr&=\left(T_mN+\g_M(m)^\omega\right)^\omega=\{0\}.
\end{aligned}
\end{equation}
We conclude that $B=\g_M\oplus (TN\oplus \g_M)^{\omega}$ has
constant rank near $m$. Further we have $F=B\cap TN=\{0\}$ since
$B_m\subset ker(d_mJ)$.

Compatibility of $\theta_D$ and $B$ holds because $F=\{0\}$.
%since  the $G$-orbits
%intersect $N$ in at most one point: this is true at $m$, since its
%$G$-orbit is contained in $J^{-1}(0)$, hence also in a neighborhood
%in $N$.
Condition \eqref{shathetaD} as well as $D:={\theta_D}|_N
\subset B$ are trivially satisfied. Condition \eqref{lieder} is
satisfied since the $G$-action preserves $\omega$.
\end{proof}

\begin{remark}\label{remaction}
1) The geometric interpretation of Prop. \ref{thetaD} applied to the
special case of  Prop. \ref{action} is the following: if the $G$
action is free and proper it is known that $M/G$ is a Poisson
manifold, whose symplectic leaves  are given by $J^{-1}(\cO)/G$ as
$\cO\subset \g^*$ ranges over all coadjoint orbits. Therefore
$\un{N}\cong N$ is a submanifold of $M/G$ which intersects
transversely the symplectic leaf $J^{-1}(0)/G$, and has such it has
a Poisson structure induced from $M/G$. This Poisson structure agree
with the one that Prop. \ref{action} induces on $N$.

%We remark that in  the case of free and proper actions the
%construction of Prop.
%\ref{action} becomes interesting. Indeed in that case,
2) In  the case that the $G$-action in  Prop.
\ref{action}  is free and proper
%at least  near $M$,
one has a dual pair $M/G\leftarrow M \rightarrow \g^*$, and Thm 8.1
of \cite{AW} says that the Poisson structure on $\un{N}$ (as in part 1)
above) is isomorphic up to sign to the one on the open subset $J(N)$
of  $\g^*$. However the   identification  $\un{N}\cong N \cong J(N)$
given by the dual pair does not preserve the Poisson structures in
general. (A sufficient condition is that $N$ is isotropic.)
\end{remark}

The following is an example for Prop. \ref{action}.
\begin{example}
Consider the action of $G=U(2)$ on $M=GL(2,\CC)$ by left multiplication, and endow $M$ with the symplectic form induced by the natural embedding in $\CC^4$. This action is Hamiltonian with moment map $J:GL(2,\CC)\rightarrow \uu^*(2)\cong \uu(2)$ given by $J(A)=\frac{1}{2i}(AA^*-I)$ \cite{Ca}.
A slice transverse to $J^{-1}(0)$ at the identity is given by
$$N :=\left\{ \left(
 \begin{smallmatrix} {x_1} & x_2+ix_3 \\ 0 & x_4 \end{smallmatrix}
\right) \right \}$$ where $x_1,x_4$ are real numbers close to $1$
and $x_2,x_3$ are close to $0$. A straightforward computation shows
that extending the coordinates $x_i$ on $N$ so that they annihilate
$[\g_M+(TN+\g_M)^{\omega}]|_N= [\g_M]|_N$ delivers the following
bracket on $N$:

\begin{eqnarray*}
 \{x_1,x_2\}=\frac{x_3}{x_1},& \{x_1,x_3\}=-\frac{x_2}{x_1},&
\{x_1,x_4\}=0\\
 \{x_2,x_3\}=1-\frac{x_4^2}{x_1^2},&
 \{x_2,x_4\}=\frac{x_3x_4}{x_1^2},&
\{x_3,x_4\}=-\frac{x_2x_4}{x_1^2}.
\end{eqnarray*}
Prop. \ref{action} states that this is a Poisson bracket.

In the new coordinates $\xi_1=\frac12x_1x_2, \xi_2=\frac12x_1x_3,
\xi_3=\frac14(x_1^2-x_2^2-x_3^2-x_4^2),
\eta=\frac14(x_1^2+x_2^2+x_3^2+x_4^2)$ the Poisson bracket is linear
and coincides  with that of $\uu^*(2)$, in agreement with Remark
\ref{remaction}.
\end{example}

\subsection{The symplectic case}

We end this section  asking when eq. \eqref{defbr} defines a \emph{symplectic} structure on the quotient $\un{N}$. We focus on the case where $M$ has a symplectic (not just Poisson) structure $\omega$.

\begin{lemma}\label{sympl}
Assume that eq. \eqref{defbr} endows $\un{N}$ with a well-defined bivector field $\un{\Pi}$.  $\un{\Pi}$ corresponds to a non-degenerate 2-form   iff
\begin{equation}\label{BBperp}
TN+B=B^{\omega}+B.
\end{equation}
In this case the 2-form on $\un{N}$ is obtained pushing down  $\omega^B\in \Omega^2(N)$ given by $\omega^B(X_1,X_2)
=\omega(X_1+b_1,X_2+b_2)$, where $b_i\in B$ are such that $X_i+b_i\in  B^{\omega}$.
\end{lemma}
\begin{proof}
From Lemma \ref{pushpull} in the Appendix it follows that $\un{\Pi}$ is invertible iff the almost Dirac structure $\iota^*(L_{\Pi}^B)$ on $N$ is the graph of a 2-form with kernel $F$. Writing out explicitly $\iota^*(L_{\Pi}^B)$ one sees that it is the graph of a 2-form iff $TN\subset B^{\omega}+B$, which in turn is equivalent to eq. \eqref{BBperp} since   eq. \eqref{tnb} holds. In this case the kernel of the 2-form is automatically $F$. This shows the equivalence claimed in the lemma.

A  computation shows that $\iota^*(L_{\Pi}^B)$ is the graph of the 2-form $\omega^B$ defined above.
\end{proof}

A simple instance of Lemma \ref{sympl}  is the case when $N$ is a symplectic submanifold of $(M,\omega)$ and $B$ is small perturbation of $TN^{\omega}$.
Then $N$ is endowed with a
non-degenerate 2-form $\omega^B$, which is intertwined with
$\iota^*\omega$ by the bundle isomorphism $TN\cong B^{\omega}$ (given
by projection along $B$).

Suppose that  $B$ can be extended locally to an integrable distribution $\theta$ on $M$ so that the $\theta$-invariant functions are closed w.r.t. the Poisson bracket. Then $\omega^B$ is a closed form, for it is just the pullback to $N$ of the symplectic form on the quotient $M/\theta$ (this is an instance of Prop. \ref{thetaD}).
In general, writing $B$ as
the graph of a bundle map $A: TN^{\circ}\cong TN^{\omega} \rightarrow TN$,
it would be interesting to spell out in terms on $A$ when $\omega^B$ is a \emph{symplectic}
structure.

\appendix

\section{}
\subsection{Algebraic interpretations}

We  provide an algebraic interpretation of Thm. \ref{gen}.
\begin{proposition}\label{alg}
Let $\cM$ be a Poisson algebra, $\cB \subset \cD$ multiplicative
subalgebras of $\cM$ and $\cI$ a multiplicative ideal of $\cM$.
Assume that the images of $\cB$ and $\cD$ under the projection $\cM
\rightarrow \cM/\cI$ are equal   and that
\begin{equation}\label{shaalg}
    \{\cB,\cI\cap \cD\}\subset \cI
\end{equation}
and
\begin{equation}\label{braalg}
 \{\cB,\cB\}\subset \cD.
\end{equation}
Then there is an induced Poisson algebra structure on
$\frac{\cB}{\cB\cap \cI}$, whose bracket is determined by the
commutative diagram
\[
\xymatrix{
 \cB \times \cB \ar[d] \ar[r]^{\{\cdot,\cdot\}}& \cD \ar[d]
\\
\frac{\cB}{\cB\cap \cI} \times \frac{\cB}{\cB\cap \cI}  \ar[r]&
\frac{\cD}{\cD\cap \cI}=\frac{\cB}{\cB\cap \cI}\\
 }.
\]
\end{proposition}

Thm. \ref{gen} is recovered setting $\cM=C^{\infty}(M)$,
%$\cB=\cB $,
$\cD=C^{\infty}(M)_D$ and $\cI=\{f\in C^{\infty}(M): \iota^*f=0\}$.
Conditions  \eqref{shaalg} and \eqref{braalg} become conditions
\eqref{sha} and \eqref{bra}
 respectively.

The proof of Prop. \ref{alg} is similar to that of Prop. \ref{gen}
and will not be given here. We just mention that condition
\eqref{shaalg} can be interpreted as ``$\cI$ behaves like an ideal
in $\cB$'', and condition \eqref{braalg} as ``$\cB$ behaves like a
Poisson subalgebra'', showing that one has  a well-defined almost
Poisson bracket on $\frac{\cB}{\cB\cap \cI}$. To show that it
satisfies the Jacobi identity one needs to use
once more both conditions.\\

\subsection{Descriptions in terms of Dirac structures}\label{dirac}

In the next proposition we interpret in terms of Dirac
structures the operation $\{\cdot,\cdot\}_{\un{N}}$ given by eq.
\eqref{defbr}.
Let $(M,\Pi)$ be a Poisson manifold,
$N\subset M$ a submanifold and $B \subset T_NM$ a   subbundle with
$F:=B\cap TN$ a regular, integrable distribution.

\begin{proposition}\label{pushpull}
Assume  that $\un{N}:=N/F$ is   smooth and
 that the prescription \eqref{defbr} gives a well-defined
bivector field on $\un{N}$, and denote by $L_{\un{N}}$ its graph.
Then the pullback of the almost Dirac structure  $L_{\un{N}}$ under
$p: N \rightarrow \un{N}$ is $\iota^*(L_{\Pi}^B)$.

 Here $L_{\Pi}^B$
is the stretching \cite{CFZ} of $L_{\Pi}=graph(\Pi)$ in direction of
$B$, defined as $[L_{\Pi}|_N\cap (T_NM \oplus B^{\circ})]+(B\oplus
0)$.
%The push-forward by $p: N \rightarrow \un{N}$ of the maximal
%isotropic subbundle $\iota^*(L_{\Pi}^)$ is  well-defined iff
%\eqref{defbr} defines an almost Poisson bracket on $\un{N}$. In this
%case it is  equal to $L_{\un{N}}$, the graph of the bivector on
%$\un{N}$.
  \end{proposition}
\begin{proof}
We will show that the Poisson algebras of admissible functions for
$\iota^*(L_{\Pi}^B)$ and $p^*(L_{\un{N}})$ match, hence the
subbundles have to agree too. Short computations using
$\sharp(TN+B)^{\circ}\subset B$ (which holds since we assume that
eq. \eqref{defbr} gives a well-defined expression) show that
$\iota^*(L_{\Pi}^B)$ is a smooth almost Dirac structure on $N$ and
  that its kernel is exactly $F$. Hence its set of admissible
  functions is $ C^{\infty}(N)_F$.
  If $f,g \in C^{\infty}(N)_F$ their $\iota^*(L_{\Pi}^B)$-bracket is
  $\langle X_{f^B}+b, dg^B \rangle$ (where
  $f^B,g^B \in C^{\infty}(M)_B$ are extensions and $b\in \Gamma(B)$ is such
  that $X_{f^B}+b\in TN$), which is equal to $\{f^B,g^B\}$.

The kernel of $p^*(L_{\un{N}})$ is clearly also $F$, and if $f,g \in
C^{\infty}(N)_F$ their $p^*(L_{\un{N}})$-bracket is
$\{f,g\}_{\un{N}}$. Using eq. \eqref{defbr} this  concludes the
proof.
\end{proof}

% Now we can use Lemma \ref{pushpull} to derive the claim made  in
% Remark \ref{rembsb}.
%\begin{corollary}\label{lemmabsb}
% Conditions \eqref{tnb} and \eqref{desc} imply that ${\rm
% rk}(B+\sharp B^{\circ})$ is constant along each leaf of the
% distribution $F$.
% \end{corollary}
%\begin{proof}
% For the sake of simplicity of exposition assume that $\un{N}$ is a
% smooth manifold\footnote{If $\un{N}$ is not smooth one has to
%replace $\un{N}$ by $\{U_i/F\}$, where $\{U_i\}$ is a suitable
% covering of $N$.}.
%  By Lemma \ref{pushpull} there exists an almost
% Dirac structure on $\un{N}$
%  whose pullback by $p:N\rightarrow \un{N}$  is $\iota^*(L_{\Pi}^B)$.
% Hence the rank of $pr_{TN}\iota^*(L_{\Pi}^B)$ must be constant along
% each leaf of $F$. A
% computation shows that
% ${\rm rk}(pr_{TN}\iota^*(L_{\Pi}^B))=dim(N)-{\rm rk}(TN+B)+
% {\rm rk}(B+\sharp B^{\circ})$,
% proving the claim.
% \end{proof}

\begin{remark}
The following statements complement Proposition \ref{pushpull} and are
proved similarly. Assume that  $\sharp(TN+B)^{\circ}\subset B$. Then
eq.  \eqref{defbr} defines a bivector field on $\un{N}$  iff
$\iota^*(L_{\Pi}^B)$ pushes forward under $p: N \rightarrow \un{N}$
(to the graph of eq. \eqref{defbr}).  If   $\iota^*(L_{\Pi}^B)$ is
integrable (i.e. if it is a Dirac structure on $N$) then it
automatically pushes forward, and therefore eq.  \eqref{defbr}
defines a Poisson structure on $\un{N}$.

Hence, assuming $\sharp(TN+B)^{\circ}\subset B$,  eq.  \eqref{defbr} defines a Poisson structure on $\un{N}$ iff $\iota^*(L_{\Pi}^B)$ is integrable. Unfortunately we were not able to express the integrability of the latter in simple
terms.
\end{remark}

\subsection{Compatibility with foliations}
We now address the question of compatibility
stated in Def. \ref{compat} and give an equivalent characterization.
By Remark \ref{remlocal} we can work locally, so in the following
we will assume that $\un{N}$ and $\un{M}:=M/\theta_D$ are smooth.

\begin{proposition}\label{equichar}
$\theta_D$ and $B$ as in Def. \ref{compat}
are compatible if and only if
\begin{eqnarray}\label{thetadescequiv}
&&\text{There exists a subbundle }
\hat{B} \text{ with }
B\subset\hat{B}\subset T_NM
\text { and }\hat{B}\cap TN=F\notag \\
&&\text{ such that }
pr: M  \rightarrow \un{M} \text{ maps }
\hat{B}\text{ to a well-defined subbundle of }T_{\un{N}}\un{M}.
\end{eqnarray}
\end{proposition}

\begin{proof}
To show the ``if'' part notice that $pr_*{\hat{B}}$
intersects trivially $T\un{N}$ (since $F \subset D$), hence any
function on $\un{N}$ can be extended to an element of
$C^{\infty}(\un{M})_{pr_*\hat{B}}$, and the pullback under $pr$ is then an
element of $C^\infty(M)_B\cap C^\infty(M)_{\theta_D}$.
Conversely, if $\theta_D$ and $B$ are compatible,
we can extend a set of coordinates on $\un{N}$ to
functions  $x_i$ on $\un{M}$ so that $pr^* x_i \in
C^{\infty}(M)_B$, and
$\hat{B}:=pr_*^{-1}(\cap{\rm ker }\dd x_i)\subset T_NM$
will satisfy the condition above.
\end{proof}

In general it is not trivial to check whether the conditions of the previous
proposition are satisfied. One can however compute easily a sufficient
condition for the compatibility in the case one can take $\hat B=B$.

To state the result we introduce
$\tilde{\Gamma}(B):= \{X\in \Gamma(TM): X|_N\subset B\}$ and
$\Gamma'(\theta_D):=\Gamma(\theta_D)\cap \tilde{\Gamma}(F)$.
Then one can prove that (\ref{thetadescequiv}) holds with
$\hat B=B$ if and only if
\begin{equation}\label{thetadescsec}
[\Gamma'(\theta_D),\tilde{\Gamma}(B)]
\subset \tilde{\Gamma}(B),
\end{equation}
which implies the compatibility of $\theta_D$ and $B$.

We conclude remarking that, given a subbundle $D$ with $F\subset D\subset B$,
locally one can always find an extension of $D$ to an involutive
distribution $\theta_D$ compatible with $B$.
% \footnote{This is proven using
% the normal bundle  $\nu N:=\tilde{B}\oplus R\rightarrow N$ defined
% in the footnote of Section \ref{extMR}. In the simple case $D=F$ the
% distribution $\theta_D$ is given by the choice of a flat $F$-connection
% on $\nu N$; in general one chooses a flat connection on a suitable
% quotient of $\nu N$.}

%\bibliography{bibliography} \bibliographystyle{amsplain}

\providecommand{\bysame}{\leavevmode\hbox to3em{\hrulefill}\thinspace}
\providecommand{\MR}{\relax\ifhmode\unskip\space\fi MR }
\providecommand{\MRhref}[2]{%
  \href{http://www.ams.org/mathscinet-getitem?mr=#1}{#2}
}
\providecommand{\href}[2]{#2}

\end{document}